\documentclass[11pt,leqno]{article}
\usepackage{amsmath,amssymb,amsthm,nicefrac,mathrsfs}
\usepackage[colorlinks]{hyperref}
\usepackage{tabu}
\usepackage{bbm}
\usepackage{graphics}
\usepackage{verbatim}
\usepackage[a4paper, margin=4cm, footskip=1cm]{geometry}
\usepackage{caption}
\usepackage[T1]{fontenc}
\usepackage{booktabs}
\usepackage{siunitx}
\usepackage{refcheck}
\newtheorem{theorem}{Theorem}[section]
\newtheorem{definition}[theorem]{Definition}
\newtheorem{lemma}[theorem]{Lemma}
\newtheorem{proposition}[theorem]{Proposition}
\newtheorem{corollary}[theorem]{Corollary}
\newtheorem{assumption}[theorem]{Assumption}
\newtheorem{remark}[theorem]{Remark}

\numberwithin{equation}{section}

\newcommand{\be}{\begin{equation}}
\newcommand{\ee}{\end{equation}}

\newcommand{\bes}{\begin{equation*}}
\newcommand{\ees}{\end{equation*}}

\def\E{\bE}

\def\bE{\mathbb{E}}

\renewcommand{\geq}{\geqslant}
\renewcommand{\leq}{\leqslant}

\def\m1{\mathbf{1}}

\allowdisplaybreaks[1]

\allowdisplaybreaks[1]
\author{Ngartelbaye Guerngar\\
	Auburn University\\
\and Erkan Nane\\ Auburn University\\}
\title{Moment bounds of a class of stochastic heat equations driven by space-time colored noise in bounded domains.
\date{}
}

\begin{document}
	
\maketitle
\begin{abstract}

	We consider the fractional stochastic heat type equation
	\begin{align*}
	\frac{\partial}{\partial t}	 u_t(x)=-(-\Delta)^{\alpha/2}u_t(x)+\xi\sigma(u_t(x))\dot{F}(t,x),\ \ \ x\in D, \ \ t>0,
	\end{align*}
	with nonnegative bounded initial condition, where $\alpha\in (0,2]$, $\xi>0$ is the noise level, $\sigma:\mathbb{R}\rightarrow\mathbb{R}$ is a globally Lipschitz function satisfying some growth conditions  and the noise term behaves in space like the Riez kernel and is possibly correlated in time and $D$ is the unit open ball centered at the origin in $\mathbb{R}^d$. When the noise term is not correlated in time, we  establish a change in the growth of the solution of these equations depending on the noise level $\xi$.
On the other hand when the noise term behaves in time like the fractional Brownian motion with index $H\in (1/2,1)$, We also derive explicit bounds leading to a  well-known intermittency property.
	\end{abstract}
\newpage
\section{Introduction}

 Stochastic Partial Differential Equations(SPDEs) have been studied a lot recently due to many challenging open problems in the area but also due to their  deep applications in disciplines that range from applied mathematics, statistical mechanics, and theoretical physics, to theoretical neuroscience, theory of complex chemical reactions [including polymer science], fluid dynamics, and mathematical finance, see for example \cite{DavCBMS} for an extensive list of literature
 devoted to the subject. On the other hand SPDEs driven by a random noise which is white in time but colored in space have increasingly received a lot of attention recently, following the foundational work of \cite{Dalang}. One difference with SPDEs driven by space-time white noise is that they can be used to model
 more complex physical phenomena which are subject to random perturbations.
Two phenomena of interest are usually observed when studying these SPDEs, "intermittency" and "phase transition". See for example \cite{BalNus}, \cite{BalNus2}, \cite{BalJolis}, \cite{BalDor}, \cite{HuNual} and \cite{fooVar} for the former and \cite{FooGuer}, \cite{FooTian}, \cite{FooNua} and \cite{XieB} for the latter.

In this article, we consider an SPDE driven by a space-time colored noise. This type of equation has received a lot of attention recently, see for example  \cite{BalNus}, \cite{BalNus2}, \cite{BalDor}, \cite{HuNual}  and the references therein. The novelty is that we assume the space to be a proper bounded open  subset of $\mathbb{R}^d$.

Consider the fractional stochastic heat equation on the open unit ball  $D$ subset of $\mathbb{R}^d, \ d\geq 1$ with zero exterior Dirichlet boundary conditions:
\begin{equation}\label{mainEq}
\begin{cases}
	\frac{\partial}{\partial t}	 u_t(x)=-(-\Delta)^{\alpha/2}u_t(x)+\xi\sigma(u_t(x))\dot{F}(t,x) \ \ \ x\in D, \ \ t>0, \\
u_t(x)=0 \ \ \ x\in D^c
\end{cases}
\end{equation}
where $\alpha\in (0,2]$, $-(-\Delta)^{\alpha/2}$ is the $L^2-$generator of a symmetric $\alpha-$stable process killed upon exiting the domain $D$. The initial condition $u_0(x)$ is a bounded and nonnegative function. The coefficient $\xi$ denotes the level of the noise, $\sigma: \mathbb{R}\rightarrow \mathbb{R}$ is a globally Lipschitz function. The mean zero Gaussian process  $\dot{F}$ is a  space-time colored noise, i.e

\begin{equation}\label{SpTiColNoise}
	\E\Big(\dot{F}(t,x)\dot{F}(s,y)\Big)=\gamma(t-s)\Lambda(x-y),
\end{equation}
where $\gamma: \mathbb{R}\rightarrow\mathbb{R_+}$ and $\Lambda:\mathbb{R}^d\rightarrow\mathbb{R_+}$ are general nonnegative and nonnegative definite(generalized) functions satisfying some integrability conditions. The Fourier transform of the latter, $\hat{\Lambda}=\mu$ is a tempered measure. We first focus  our attention on the case where the noise term is uncorrelated in time.
 	
	The objective of this paper is to provide lower and upper bounds for the moments of the stochastic fractional heat equation \eqref{mainEq}. But first, let us define some terms and expressions we will use in this paper.

\begin{definition}
Assume $\gamma=\delta_0$. Following \cite{WalshB}, a random field $\{u_t(x)\}_{t>0, x\in D}$ is called a mild solution of \eqref{mainEq} in the  Walsh-Dalang sense if
\begin{enumerate}
	\item $u_t(x)$ is jointly measurable in $t\geq 0$ and $x\in D$;
	\item for all $(t,x)\in \mathbb{R}_+\times D$, the stochastic integral $\int_0^t\int_{D}p_D(t-s,x,y)\sigma\big(u_s(y)\big)F(dy,ds)$ is well-defined in $L^2(\Omega)$; 	
	Moreover, $\sup\limits_{t>0}\sup\limits_{x\in D}\E|u_t(x)|^p<\infty, \ \ \text{for all }\ p\geq 2.$;
	\item The following integral equation holds in $L^2(\Omega)$:
	
	\begin{equation}\label{mlSol}
	u_t(x)= \big(\mathcal{G}u_0\big)_t(x)+\xi\int_0^t\int_{D}p_D(t-s,x,y)\sigma\big(u_s(y)\big)F(dy,ds),
	\end{equation}
	where  $$\big(\mathcal{G}u_0\big)_t(x):= \int_D p_D(t,x,y)u_0(y)dy$$ and $p_D(t,x,y)$ denotes the Dirichlet  heat kernel of the stable L\'evy process.
	It is the transition density of the stable L\'evy process killed in the
	exterior of $D$. Please refer to Section \ref{Sec2} for a short description of the latter.
\end{enumerate}

\end{definition}

When $\gamma=\delta_0,$ following Dalang \cite{Dalang}, it is well-known that if the spectral measure satisfies the following condition:

\begin{equation}\label{DalCond}
	\int_{\mathbb{R}^d}\frac{\mu(\zeta)}{1+|\zeta|^\alpha}<\infty,
\end{equation}
then there exists  a unique   random field solution of \eqref{mainEq}. Please refer to Section \ref{Appdx}  for the proof of existence and uniqueness of a random field solution in this case.

  Some examples of space correlation functions satisfying condition \eqref{DalCond} include

 \begin{itemize}
 	\item \textbf{Space-time white noise:} $\Lambda=\delta_0$ in which case $\mu(d\zeta)= d\zeta$ and \eqref{DalCond} holds only when $\alpha>d$ which implies $d=1$ and $1<\alpha\leq 2$.
 	\item\textbf{ Riez Kernel:} $\Lambda(x)= |x|^{-\beta}, \ \ 0<\beta<d$. Here  $\mu(d\zeta)= c|\zeta|^{-(d-\beta)} d\zeta$ and \eqref{DalCond} holds whenever $\beta<\alpha$.
 	\item \textbf{Bessel kernel:} $\Lambda(x)= \int_0^\infty y^{\frac{\eta-d}{2}} e^{-y} e^{-\frac{|x|^2}{4y}}dy$. $\mu(d\zeta)= c(1+|\zeta|^2)^{-\frac{\eta}{2}}d\zeta$ and \eqref{DalCond} implies $\eta>d-\alpha$.
 	\item \textbf{Fractional Kernel:} $\Lambda(x)= \prod_{i=1}^{d}|x_i|^{2H_i-2}$. $\mu(\zeta)= c\prod_{i=1}^{d}|x_i|^{1-2H_i}d\zeta$ and \eqref{DalCond} holds  whenever $\sum_{i=1}^{d}H_i> d-\frac{\alpha}{2}.$
 	 \end{itemize}
  We refer the interested reader to \cite{fooVar1} for more examples of such functions.
 	We now turn our attention on the case where the noise term is also  correlated in time.
 	
 	\begin{definition}\label{defskhd}
 	Assume $\sigma=Id$, the identity map.	An adapted random field  $\{u_t(x)\}_{t>0, x\in D}$ such that $\E[u_t(x)]^2<\infty$ for all $(t,x)$ is a mild solution to \eqref{mainEq} in the Skorohod sense if for any $(t,x)\in \mathbb{R}_+\times D$, the process $\{ p_D(t-s,x,y)u_s(y)\textbf{1}_{[0,t]}(s): s\geq 0, \ y\in D\}$ is Skorohod integrable and the following integral equation holds:
 \begin{equation}\label{Skhd}
 u_t(x)= \big(\mathcal{G}u_0\big)_t(x)+\xi\int_0^t\int_{D}p_D(t-s,x,y)u_s(y) F(\delta s,\delta y).
 \end{equation}

 	\end{definition}

It is well-known that a unique  mild solution \eqref{Skhd} exists in the Skorohod sense provided that the time correlation $\gamma$ is locally integrable and the space correlation $\Lambda$ satisfies condition \eqref{DalCond}.
When handling the mild solution in the Skorohod sense, we shall make use of the Wiener-chaos expansion.

 Recall that the covariance given by \eqref{SpTiColNoise} is a mere  formal notation. Let $C_0^\infty(\mathbb{R}_+\times \mathbb{R}^d)$ be the space of test functions on $\mathbb{R}_+\times \mathbb{R}^d.$ Then on a complete probability space $(\Omega, \mathcal{F}, P)$, we consider a family of centered Gaussian random variables indexed by the test function $\Big\{ F(\varphi), \varphi\in C_0^\infty(\mathbb{R}_+\times \mathbb{R}^d) \Big\}$  with covariance

 	\begin{equation}\label{CovFl}
 \E[\dot{F}(\varphi)\dot{F}(\psi)]= \int_{\mathbb{R}_+^2\times \mathbb{R}^{2d}}\varphi(t,x)\psi(s,y) \gamma(t-s)\Lambda(x-y)dxdydtds.
 \end{equation}
 	We write equation \eqref{CovFl} formally as \eqref{SpTiColNoise}.
Let $\mathcal{H}$ be the completion of $C_0^\infty(\mathbb{R}_+\times \mathbb{R}^d)$ with respect to the inner product

\begin{equation*}
\langle \varphi, \psi\rangle_{\mathcal{H}}= \int_{\mathbb{R}_+^2\times \mathbb{R}^{2d}}\varphi(t,x)\psi(s,y)\gamma(t-s)\Lambda(x-y)dxdydtds.
\end{equation*}
 	
The mapping $\varphi\mapsto F(\varphi)\in L^2(\Omega)$ is an isometry which can be extended to $\mathcal{H}$. We denote this map by

\begin{align*}
F(\varphi)=  \int_{\mathbb{R}_+\times \mathbb{R}^{d}}\varphi(t,x)	F(dt,dx), \ \ \varphi\in \mathcal{H}.
\end{align*}	
 Note that if $\varphi, \psi\in \mathcal{H},$

 	\begin{equation*}
 \E[\dot{F}(\varphi)\dot{F}(\psi)]= \langle \varphi, \psi\rangle_{\mathcal{H}}.
 \end{equation*}	
 	
Furthermore, $\mathcal{H}$ contains the space of mesurable functions $\varphi$ on $\mathbb{R}_+\times \mathbb{R}^d$ such that

\begin{align*}
\int_{\mathbb{R}_+^2\times \mathbb{R}^{2d}}|\varphi(t,x)\varphi(s,y)|\gamma(t-s) \Lambda(x-y)dxdydtds<\infty.
\end{align*}
For $n\geq 0$, denote by $\textbf{H}_n$ the $n^{th}$ Wiener-chaos of $F$. Recall that $\textbf{H}_0$ is just $\mathbb{R}$ and for $n\geq 1$, $\textbf{H}_n$ is the closed linear subspace of $L^2(\Omega)$ generated by the random variables\\
$\{H_n(F(h)), \ h\in \mathcal{H}, \ \|h\|_\mathcal{H}=1 \}$ where $H_n$ is the $n^{th}$ Hermite polynomial. 	For $n\geq 1$, we denote  by $\mathcal{H}^{\otimes n}(\text{resp.} \mathcal{H}^n)$ the $n^{th}$ tensor product (resp. the $n^{th}$ symmetric tensor product) of $\mathcal{H}$. Then, the mapping $I_n(h^{\otimes n})= H_n(F(h))$ can be extended to a linear isometry between $\mathcal{H}^{ n}$ (equipped with the modified norm $\sqrt{n!}\|.\|_{\mathcal{H}^{\otimes n}}$) and $\textbf{H}_n$, see for example \cite{DNua} and \cite{NunOksend} and the references therein.

 Consider now a random variable $X\in L^2(\Omega)$ measurable with respect to the $\sigma-$field $\mathcal{F}^F$ generated by $F.$ This random variable can be expressed as
 \begin{equation*}
 	X=\E[X]+ \sum_{n=1}^\infty I_n(f_n),
 \end{equation*} 	
 where the series converges in $L^2(\Omega)$ and the elements $f_n\in\mathcal{H}^n, \ n\geq 1$ are determined by $X$. This identity is known as the Wiener-chaos expansion.  Please refer to \cite{DNua} and \cite{NunOksend} for a complete description on  the matter.
\newpage
  We will need the following assumptions:
	\begin{assumption}\label{hyp1}
		$\gamma: \mathbb{R}\rightarrow \mathbb{R_+}$ is locally integrable.
	\end{assumption}

\begin{assumption}\label{hyp2}
	There exist constants $C_1$ and $C_2$ and $0<\beta<\alpha\wedge d$ such that for all $x\in \mathbb{R}^d,$
	$$
	C_1|x|^{-\beta}\leq \Lambda(x)\leq C_2|x|^{-\beta}.
	$$
\end{assumption}

\begin{assumption}\label{hyp3}
	There exist positive constants $l_\sigma$ and $L_\sigma$ such that for all $x\in \mathbb{R}^d,$
	$$
	l_\sigma|x|\leq \sigma(x)\leq L_\sigma|x|.
	$$
\end{assumption}
\begin{assumption}\label{hyp4}
	There is $\epsilon\in \big(0,\frac{1}{2} \big)$ such that
	$$
	\inf\limits_{x\in D_\epsilon} u_0(x)>0,
	$$
	where $ D_\epsilon:= B_{1-\epsilon}(0)$
\end{assumption}
Throughout the remainder of this paper,    $\alpha\in (0,2]$, the letter C or c with or without subscript(s) denotes a constant with no major importance to our study,  assumption \ref{hyp2} and assumption \ref{hyp4} hold unless stated otherwise. $B_R(0)$ represents the open ball of radius $R$ centered at the origin. We are now ready to state our main results.
\begin{theorem}\label{thm1}
	Assume $\gamma= \delta_0$ and $\sigma$ satisfies assumption \ref{hyp3} . Then for all $t>0$ and $p\geq 2$, there exist positive constants $c_1, c_2(\alpha, \beta, d, l_\sigma), C_1$ and $C_2(\alpha,\beta, d,  L_\sigma)$ such that for all $\xi>0$ and $\delta>0$,
	\begin{align*}
c_1^p e^{pt\Big(c_2 \xi^{\frac{2\alpha}{\alpha-\beta}}-\mu_1\Big)}	\leq 	\inf\limits_{x\in D_\epsilon}\E|u_t(x)|^p\leq \sup\limits_{x\in D}\E|u_t(x)|^p\leq  C_1^p e^{pt\Big( C_2\xi^{\frac{2\alpha}{\alpha-\beta}}z_p^{\frac{2\alpha}{\alpha-\beta}}-(1-\delta)\mu_1  \Big)},
	\end{align*}
	where $z_p$ is the constant in the Burkh\"older-Davis-Gundy's inequality.
\end{theorem}
This theorem shows  that the rate at which the moments  of the solution to equation \eqref{mainEq} exponentially  grow or decay depends explicitly on the non-local operator $-(-\Delta)^{\alpha/2}$, the noise level $\xi$ and the noise term via the quantity $\xi^\frac{2\alpha}{\alpha-\beta}$.
This result provides an extension to \cite{ENua} where the author used equation \eqref{mainEq} with $\sigma=Id$,  an essential assumption when using the Wiener-Chaos expansion in the proofs. However, the proof we provide for this theorem uses a different argument. This theorem also provides an extension to \cite{FooLiu} where similar bounds were obtained but only for the second moments of the solution to equation \eqref{mainEq}. While showing explicitly the dependence of moments of solution of equation \eqref{mainEq} with the noise level $\xi$,  it also implies that there exist $\xi_0(p)>0$ such that for all $\xi<\xi_0$ and $x\in D,$

$$
-\infty< \limsup\limits_{t\rightarrow \infty}\frac{1}{t}\log \E|u(t,x)|^p <0,
$$
 and  there exists $\xi_1(p)$ such that for all $\xi>\xi_1$ and for all $\epsilon>0$,  $x\in D_\epsilon$,
$$
0<\liminf\limits_{t\rightarrow \infty}\frac{1}{t}\log \E|u(t,x)|^p<\infty.
$$
These results  were proved in \cite[for the case $\alpha=2$]{FooNua} and \cite[for  $0<\alpha<2$]{FooGuer} but without showing  the explicit dependence of the moments  on $\xi.$

The results provided in this Theorem also lead to yet another phenomenon known as intermittency.
Define the $p^{th}$ upper Liapounov moment of the random field  $u:=\{ u_t(x)\}_{t>0, x\in D}$ at $x_0\in D$ as
\begin{align*}
\overline{\gamma}(p):= \limsup\limits_{t\rightarrow\infty}\frac{1}{t}\log\E|u_t(x)|^p \ \ \text{for all } \ p\in (0,\infty).
\end{align*}
Following \cite{fooVar}, the random field $u$ is said to be \textit{weakly intermittent} if:

\begin{align*}
\text{for all} \ x\in D, \ \  \overline{\gamma}(2)>0 \ \text{and} \  \overline{\gamma}(p)<\infty \ \ \text{for all } \ p\in (2,\infty).
\end{align*}
It is said to be \textit{fully intermittent} if:

\begin{align*}
p\mapsto \frac{\overline{\gamma}(p)}{p}	 \ \ \text{is strictly increasing for all } \ p\geq 2 \ \ \text{and} \ x\in D.
\end{align*}

It is also known that weak-intermittency can sometimes imply full intermittency, see for example \cite{fooVar} and the references therein. In Theorem \ref{thm1}, when $\xi< \Big(\mu_1/C(p,\delta)\Big)^{\frac{\alpha-\beta}{2\alpha}}$, the solution $u$ is not weakly-intermittent. However, quite the opposite situation occurs for the same  random field $u$ when $\xi>\Big(\mu_1/C_1(p)\Big)^{\frac{\alpha-\beta}{2\alpha}}.$

The next result is concerned with the space-time colored noise case.

	\noindent One of the time correlation functions that has received a lot of attention lately is the correlation function of the  so-called  fractional  Brownian motion (of index $H$) i.e
	\begin{equation}\label{TimCorr}
	\gamma(r)= C_H|r|^{2H-2}, \ \ \ for \ \ H\in \big(1/2, 1\big) \ \text{and} \ C_H= H(2H-1).
	\end{equation}
	Note that this function clearly satisfies Assumption \ref{hyp1}. We refer the interested reader to \cite{BalNus} and the references therein for more information about this function

\begin{theorem}\label{thm2}
	Assume $\sigma(x)=x$ and $\gamma$ is given by \eqref{TimCorr}. Then for all $t>0$ and $p\geq 2$, there exist positive  constants $\overline{c}_1, \overline{c}_2(\alpha, \beta), \
\overline{C}_1$ and $\overline{C}_2(\alpha, \beta)$  such that for all $\xi>0$ and $\delta>0$,
{\small
			\begin{align*}
		\overline{c}_1^p e^{p\Big(\overline{c}_2 t^{\frac{2H\alpha-\beta}{\alpha-\beta}} \xi^{\frac{2\alpha}{\alpha-\beta}}-\mu_1t\Big)}	\leq 	\inf\limits_{x\in D_\epsilon}\E|u_t(x)|^p\leq \sup\limits_{x\in D}\mathbb{E}|u_t(x)|^p\leq \overline{C}_1^p e^{\overline{C}_2 p\Big((p-1)^{\frac{\alpha}{\alpha-\beta}} t^{\frac{2H\alpha-\beta}{\alpha-\beta}} \xi^{\frac{2\alpha}{\alpha-\beta}}-(\mu_1-\delta)t\Big)}.
		\end{align*}
	}

\end{theorem}
 Though these bounds might not be very sharp, to the best of our knowledge, this is the first paper ever to examine  the moments of the solution of  SPDEs driven by such type of noise in bounded domains. Notice again the dependence of moments with the noise level.

\begin{remark}
Theorem \ref{thm2} also holds for more general time correlation functions satisfying Assumption \ref{hyp1}. Please refer to the proof of Theorem \ref{thm2} for more details.
\end{remark}
When equation \eqref{mainEq} is driven by a noise correlated in time, we observe   a different notion of weak intermittency. This is obtained by a modification of the Lyapunov exponent. For $\rho>0$  and $x\in D$ we define the modified upper Lyapunov exponent (of index $\rho$) by

\begin{align*}
\overline{\gamma}_\rho(p):= \limsup\limits_{t\rightarrow\infty}\frac{1}{t^\rho}\log\E|u_t(x)|^p \ \ \text{for all } \ p\in (0,\infty).
\end{align*}
$u$ is weakly $\rho-$ intermittent if
\begin{align*}
\text{for all} \ x\in D, \ \  \overline{\gamma}_\rho(2)>0 \ \text{and} \  \overline{\gamma}_\rho(p)<\infty \ \ \text{for all } \ p\in (2,\infty).
\end{align*}
It is said to be \textit{  fully $\rho-$intermittent} if:

\begin{align*}
p\mapsto \frac{\overline{\gamma}_\rho(p)}{p}	 \ \ \text{is strictly increasing for all } \ p\geq 2 \ \ \text{and} \ x\in D.
\end{align*}
Theorem \ref{thm2} shows in fact that $u$ is $\rho-$weakly-intermittent for all $\xi>0$ since $\rho:=\frac{2H\alpha-\beta}{\alpha-\beta}>1.$ Similar result was obtained in \cite{BalNus} and \cite{BalNus2} and \cite{HuNual} for the case $\alpha=2$, but the authors worked on the entire Euclidean space $\mathbb{R}^d$. Furthermore, in \cite{BalNus},  the bounds were not obtained  for all values of $t>0$. In addition, the value of $\rho=\frac{2H\alpha-\beta}{\alpha-\beta}$ matches the value  of $\rho$ found in \cite{BalNus} and $\rho_h$ found in \cite{BalNus2} for $\alpha=2$.

\begin{corollary}
Under the assumptions in Theorem \ref{thm2},
 we have
 $$0<\liminf_{t\rightarrow \infty}\frac{1}{t^{\rho}} \log\mathbb{E}|u_t(x)|^p\leq \limsup_{t\rightarrow \infty}\frac{1}{t^{\rho}} \log\mathbb{E}|u_t(x)|^p<\infty  \ \ \ \text{for all} \ x\in D_\epsilon,$$
  where $\rho=\frac{2H\alpha-\beta}{\alpha-\beta}>1.$
\end{corollary}

\vspace{0.5cm}

The rest of the paper is organized as follows: in section 2, we provide several estimates needed for the proofs of our results; section 3 is devoted to the proofs of our main results and this paper ends with an Appendix where useful results from other authors are compiled.

\section{Preliminaries}\label{Sec2}
The Dirichlet heat kernel  will play a major role in the proof of our results. Here we give a few details about it.
We define the "killed process":
\begin{equation*}
X_t^D=
\begin{cases}
X_t \ \ \ t<\tau_D \\
0 \ \ \ t\geq \tau_D,
\end{cases}
\end{equation*}
where $\tau_D= \inf \{ t>0: X_t\notin D\}$ is the first exiting time.\\
Define $$r^D(t,x,y):= \E^x[p(t-\tau_D, X_{\tau_D}, y); \tau_D<t],$$ then \\
$$p_D(t,x,y)= p(t,x,y)-r^D(t,x,y),$$
where $p(t,.,.)$ is the transition density of the "unkilled process" $X_t$. Note that $p(t,x,y)$ is also written $p(t,x-y)$ in some literature.

 When $\alpha=2$,  $X_t$ corresponds to a Brownian motion (Wiener process)$(B_t)_{t\geq 0}$ with variance $2t$,
and in this case $p(t,.,.)$ is explicitely given by

\begin{equation}\label{ExplP}
p(t,x,y)=(4\pi t)^{-d/2}e^{-\frac{|x-y|^2}{4t}} \ \ \ \text{for all} \ x,y\in \mathbb{R}^d.
\end{equation}

 When $\alpha\in(0,2)$, then $X_t$ coincides with an $\alpha$-stable L\'evy process given by $X_t = B_{S_t}$,  where $(S_t)_{t\geq 0}$
 is an $\alpha/2$-stable subordinator with L\'evy measure

 $$\nu(dx) = \frac{\alpha/2}{\Gamma(1-\alpha/2)}x^{-1-\alpha/2}\textbf{1}_{\{x>0\}} dx.$$

No explicit expression is known for $p(t,.,.)$ in this case, but
the following approximation holds:
\begin{equation}\label{Pbds}
	C_1 \min\Bigg( t^{-d/\alpha}, \frac{t}{|x-y|^{\alpha+d}}\Bigg)\leq p(t,x,y)\leq C_2\min\Bigg( t^{-d/\alpha}, \frac{t}{|x-y|^{\alpha+d}}\Bigg)
\end{equation}
for some positive constants $C_1$ and $C_2$. See for example \cite{ChenKim} and the references therein. One  important property  of the heat kernel $p(.)$ is the Chapman-Kolmogorov identity (also known as the semigroup property), i.e

\begin{equation}\label{ChapKv}
\int_{\mathbb{R}^d}  p(t,x,z)p(s,y,z)dz= p(t+s, x,y) \ \ \text{for all }  \ x,y\in \mathbb{R}^d \ \ \text{and} \ s,t>0.
\end{equation}
It is  an easy fact that $p_D(.)$ also satisfies the Chapman-Kolmogorov identity. Recall that the Dirichlet  heat kernel $p_D(t,x,y)$ has the spectral decomposition
$$
p_D(t,x,y)=\sum_{n=1}^\infty e^{-\mu_n t}\phi_n(x)\phi_n(y), \ \ \ \text{for all} \ \ x,y \in D, \ \ t>0,
$$
where $\{ \phi_n\}_{n\geq 1}$ is an orthonormal basis of $L^2(D)$ and $0<\mu_1\leq \mu_2\leq ...\leq \mu_n\leq ...$ is a sequence of positive numbers satisfying, for all $n\geq 1:$

$$
\begin{cases}
&-(-\Delta)^{\alpha/2}\phi_n(x)=-\mu_n\phi_n (x) \ \ \ \ x\in D \\
&\phi_n(x)= 0 \hspace{3cm}  \ \ \ x\in D^c.
\end{cases}
$$

It is well-known that
\begin{equation}
	c_1 n^{\alpha/d}\leq \mu_n\leq c_2 n^{\alpha/d}
\end{equation}
for some constants $c_1, c_2>0$. See for example \cite[Theorem 2.3]{BluToor},  for more details. Moreover by \cite[ Theorem 4.2]{ChenSong}, for all $x\in D$,
\begin{equation}\label{PhiBonds}
	c^{-1}(1-|x|)^{\alpha/2}\leq \phi_1(x)\leq c(1-|x|)^{\alpha/2}, \ \ \ \text{for some } \ \ c>1.
\end{equation}
For example when $\alpha=2$, and $d=1$,  i.e $D=(-1,1)$, we get for $n=1,2,...$  $$\phi_n(x)=\sin\Big[\frac{n\pi }{2}(x+1)\Big] \ \ \text{ and} \ \ \mu_n=\Big(\frac{n\pi}{2}\Big)^2.$$
\vspace{0.5cm}

We shall need the following estimates to prove our main results. The first two follow from applications of  Theorems \ref{Riah} and \ref{chenkim}.
\begin{proposition}\label{lwbpD}
	Fix $\epsilon\in(0,\frac{1}{2})$. Then for any $x,y\in D_\epsilon$ such that $|x-y|<t^{1/\alpha}$, we have
	$$
	p_D(t,x,y)\geq C t^{-d/\alpha}e^{-\mu_1t} \ \ \ \text{for all} \ t>0
	$$
	and some positive constant $C$.
\end{proposition}

\begin{proof}
	We first prove the Lemma for $\alpha=2$.  Assume $|x-y|<\sqrt{t}$. We apply Theorem \ref{Riah} to get
	\begin{align*}
		p_D(t,x,y)\geq & C_1\min\Bigg( 1, \frac{\phi_1(x)\phi_1(y)}{1\wedge t}\Bigg)
		e^{-\mu_1 t}\frac{e^{-c_1\frac{|x-y|^2}{t}}}{1\wedge t^{d/2}}\\
		\geq &C_2e^{-\mu_1t}\Bigg\{\min\Bigg( 1, \frac{c^{-2}\epsilon^2}{t}
		\Bigg)\frac{e^{-c_1\frac{|x-y|^2}{t}}}{ t^{d/2}}\textbf{1}_{\{ t<1\}}
		+\min\Big( 1, c^{-2}\epsilon^2\Big)e^{-c_1\frac{|x-y|^2}{t}}
		 \textbf{1}_{\{ t\geq 1\}}\Bigg\}\\
		=& C_2e^{-\mu_1t}\Bigg\{ \frac{e^{-c_1\frac{|x-y|^2}{t}}}{ t^{d/2}}
		\textbf{1}_{\{ t<c^{-2}\epsilon^2\}}
		+ c^{-2}\epsilon^2\frac{e^{-c_1\frac{|x-y|^2}{t}}}{ t^{1+d/2}}
		\textbf{1}_{\{ c^{-2}\epsilon^2\leq t<1\}}\\
		& \hspace{7.2cm}+\min\Big( 1, c^{-2}\epsilon^2\Big)e^{-c_1\frac{|x-y|^2}{t}}
		\textbf{1}_{\{ t\geq 1\}}\Bigg\}\\
		\geq&  C_3e^{-\mu_1t}t^{-d/2}\Bigg\{ \textbf{1}_{\{ t<c^{-2}\epsilon^2\}}+ \frac{c^{-2}\epsilon^2}{t}\textbf{1}_{\{ c^{-2}\epsilon^2\leq t<1\}}+ c(\epsilon) t^{d/2}
		\textbf{1}_{\{ t\geq 1\}}\Bigg\}\\
		\geq&  C_4e^{-\mu_1t}t^{-d/2}.
	\end{align*}
	Note the use of \eqref{PhiBonds} in the second inequality above since $x,y\in D_\epsilon$. This proves the inequality for $\alpha=2$.

	Now suppose $0<\alpha<2$. Assuming $|x-y|<t^{1/\alpha}$, we apply Theorem \ref{chenkim} to get
	
	\begin{align*}
	p_D(t,x,y)\geq &	 C_{1}e^{-\mu_1t}\Bigg[ \min\Big(1, \frac{\phi_1(x)}{\sqrt{t}}\Big)\min\Big(1, \frac{\phi_1(y)}{\sqrt{t}}\Big)\min\Bigg(t^{-d/\alpha}, \frac{t}{|x-y|^{\alpha+d}}\Bigg)\textbf{1}_{\{t<1\}}\\
	 &\hspace{9.4cm}+\phi_1(x)\phi_1(y) \textbf{1}_{\{t\geq 1\}}\Bigg]\\
	\geq &	 C_{2}e^{-\mu_1t}\Bigg\{  \min\Big(1, \frac{c^{-2}\epsilon^\alpha}{t}\Big)\min\Bigg(t^{-d/\alpha}, \frac{t}{|x-y|^{\alpha+d}}\Bigg)\textbf{1}_{\{t<1\}}
	+ c^{-2}\epsilon^\alpha \textbf{1}_{\{t\geq 1\}}\Bigg\}\\
	\geq &	 C_{3}e^{-\mu_1t}\Bigg\{ \min\Bigg(t^{-d/\alpha}, \frac{t}{|x-y|^{\alpha+d}}\Bigg)\textbf{1}_{\{t<c^{-2}\epsilon^\alpha\}}\\
	& \hspace{2.8cm}+\frac{c^{-2}\epsilon^\alpha}{t}\min\Bigg(t^{-d/\alpha}, \frac{t}{|x-y|^{\alpha+d}}\Bigg)\textbf{1}_{\{c^{-2}\epsilon^\alpha\leq t<1\}}
	+ c^{-2}\epsilon^\alpha \textbf{1}_{\{t\geq 1\}}\Bigg\}\\
	= &	 C_{3}e^{-\mu_1t}t^{-d/\alpha}\Bigg\{ \min\Bigg[1, \Bigg(\frac{t^{1/\alpha}}{|x-y|}\Bigg)^{\alpha+d}\Bigg]
	\textbf{1}_{\{t<c^{-2}\epsilon^\alpha\}}\\
	& \hspace{2.3cm}+\frac{c^{-2}\epsilon^\alpha}{t}
	\min\Bigg[1, \Bigg(\frac{t^{1/\alpha}}{|x-y|}\Bigg)^{\alpha+d}\Bigg]
	\textbf{1}_{\{c^{-2}\epsilon^\alpha\leq t<1\}}
	+ c^{-2}\epsilon^\alpha t^{d/\alpha}\textbf{1}_{\{t\geq 1\}}\Bigg\}\\
	&	=  C_{3}e^{-\mu_1t}t^{-d/\alpha}\Bigg\{\textbf{1}_{\{t<c^{-2}\epsilon^\alpha\}}
	+\frac{c^{-2}\epsilon^\alpha}{t}\textbf{1}_{\{c^{-2}\epsilon^\alpha\leq t<1\}}+ c^{-2}\epsilon^\alpha t^{d/\alpha}\textbf{1}_{\{t\geq 1\}} \Bigg\}\\
	&\geq 	 C_{4}e^{-\mu_1t}t^{-d/\alpha}.
	\end{align*}
	Again note the use of \eqref{PhiBonds} in the second inequality above. This concludes the proof.
\end{proof}

\begin{lemma}\label{lm2}
For all $\delta>0,$ there exists $c_2(\delta)>0$ such that for all $x,\ w\in D$ and $s,t>0,$
$$
\int_{D\times D} p_D(t,x,y)p_D(s,w,z)\Lambda(y-z) dydz\leq
c_2 e^{-(1-\delta)\mu_1(t+s) }\Big(s+ t\Big)^{-\beta/\alpha}
$$	
\end{lemma}
\begin{proof}
	As usual, we first prove the result for $\alpha=2$ .
	By Theorem \ref{Riah}, we have
	
	\begin{align*}
	  &\int_{D^2} p_D(t,x,y)p_D(s,w,z)\Lambda(y-z) dydz\\
	 & \leq C_1 e^{-\mu_1(t+s)}\int_{D^2}\min\Bigg( 1, \frac{\phi_1(x)\phi_1(y)}{1\wedge t}\Bigg)\min\Bigg( 1, \frac{\phi_1(w)\phi_1(z)}{1\wedge s}\Bigg)\\
	 & \hspace{8.3cm}\times \frac{e^{-c_1\frac{|x-y|^2}{t}}}{1\wedge t^{d/2}}\frac{e^{-c_2\frac{|w-z|^2}{s}}}{1\wedge s^{d/2}}\Lambda(y-z) dydz\\
	 &\leq  C_2 e^{-\mu_1(t+s)} \Bigg\{ \int_{\mathbb{R}^d\times\mathbb{R}^d} p(t,x,y)p(s,w,z) \Lambda(y-z)dydz\textbf{1}_{ \{ t< 1, s< 1 \} }\\
	 & \hspace{.6cm}+\int_{\mathbb{R}^d} p(t,x,y)\Lambda(y-z)dy\textbf{1}_{ \{ t< 1, s\geq 1 \} }
	 + \int_{\mathbb{R}^d} p(s,w,z) \Lambda(y-z)dz\textbf{1}_{ \{ t\geq 1, s< 1 \} }+c\textbf{1}_{ \{ t\geq 1, \ s\geq 1 \} } \Bigg\} \\
	 &=  C_2 e^{-\mu_1(t+s)} \Bigg\{ \int_{\mathbb{R}^{2d}} p(t+s,x-w,y) \Lambda(y)dy\textbf{1}_{ \{ t<1, s< 1 \} }+
	 \int_{\mathbb{R}^d} p(t,x,y)\Lambda(y-z)dy\textbf{1}_{ \{ t< 1, s\geq 1 \} }\\
	 & \hspace{6.7cm}+ \int_{\mathbb{R}^d} p(s,w,z) \Lambda(y-z)dz\textbf{1}_{ \{ t\geq 1, s< 1 \} }+c\textbf{1}_{ \{ t\geq 1, \ s\geq 1 \} } \Bigg\} \\
	 & \leq  C_3 e^{-\mu_1(t+s)} \Bigg\{ c_1 (t+s)^{-\frac{\beta}{2}}\textbf{1}_{ \{ t< 1, s< 1 \} }\ +
	 c_2t^{-\frac{\beta}{2}}\textbf{1}_{ \{ t< 1, s\geq 1 \} } \
	 + c_3 s^{-\frac{\beta}{2}}\textbf{1}_{ \{ t\geq 1, s< 1 \} }\ +  c\textbf{1}_{ \{ t\geq 1, \ s\geq 1 \} } \Bigg\} \\
	 & = C_3 e^{-\mu_1(t+s)}(t+s)^{-\frac{\beta}{2}} \Bigg\{c_1\textbf{1}_{ \{ t< 1, s< 1 \} }+c_2\Big(1+\frac{s}{t}\Big)^{\frac{\beta}{2}}\textbf{1}_{ \{ t< 1, s\geq 1 \} }+c_3\Big(1+\frac{t}{s}\Big)^{\frac{\beta}{2}}\textbf{1}_{ \{ t\geq 1, s< 1 \} }\\
	 & \hspace{11.4cm} +(t+s)^{\beta/2}\textbf{1}_{ \{ t\geq 1, \ s\geq 1 \} }   \Bigg\}\\
	 & \hspace{1cm}\leq C_5 e^{-(1-\delta)\mu_1(t+s)}(t+s)^{-\beta/2} \ \ \text{for all} \ \delta>0.
	\end{align*}
	Note the use of \eqref{ExplP} in the second inequality, the Chapman-Kolmogorov identity \eqref{ChapKv} in the first integral in the third inequality and Proposition \ref{TimSc} in the fourth inequality. 
	
	\vspace{0.5cm}

The proof for the case $0<\alpha<2$ follows  a very  similar argument. By Theorem \ref{chenkim}, we have
	\begin{align*}
		 & \int_{D^2} p_D(t,x,y)p_D(s,w,z)\Lambda(y-z) dydz \\
		 & \leq  C_1 e^{-\mu_1(t+s)}\Bigg\{ \int_{ D^2}\min\Bigg(t^{-\frac{d}{\alpha}}, \frac{t}{|x-y|^{\alpha+d}}\Bigg) \min\Bigg(s^{-\frac{d}{\alpha}}, \frac{s}{|w-z|^{\alpha+d}}\Bigg)\\
		  & \quad\times \Lambda(y-z)dydz\textbf{1}_{ \{ t< 1, s< 1 \} }
		+\int_{D^2} \min\Bigg(t^{-d/\alpha}, \frac{t}{|x-y|^{\alpha+d}}\Bigg)\Lambda(y-z)dydz\textbf{1}_{ \{ t< 1, s\geq  1 \} }\\
		 &\hspace{3cm}+\int_{D^2} \min\Bigg(s^{-d/\alpha}, \frac{t}{|w-z|^{\alpha+d}}\Bigg)\Lambda(y-z)dydz\textbf{1}_{ \{ t\geq  1, s< 1 \} }+c \textbf{1}_{ \{ t\geq  1, s\geq  1 \} } \Bigg\}\\
		 & \leq  C_2 e^{-\mu_1(t+s)} \Bigg\{ \int_{\mathbb{R}^{2d}} p(t,x,y)p(s,w,z) \Lambda(y-z)dydz\textbf{1}_{ \{ t<1, s< 1 \} }\\
		 & \quad+ \int_{\mathbb{R}^d} p(t,x,y)\Lambda(y-z)dy\textbf{1}_{ \{ t< 1, s\geq 1 \} }
		  + \int_{\mathbb{R}^d} p(s,w,z) \Lambda(y-z)dz\textbf{1}_{ \{ t\geq 1, s< 1 \} }+c\textbf{1}_{ \{ t\geq 1, \ s\geq 1 \} } \Bigg\} \\
& =  C_2e^{-\mu_1(t+s)} \Bigg\{ \int_{\mathbb{R}^{2d}} p(t+s,x-w,y) \Lambda(y)dy\textbf{1}_{ \{ t<1, s<1 \} }+
\int_{\mathbb{R}^d} p(t,x,y)\Lambda(y-z)dy\textbf{1}_{ \{ t< 1, s\geq  1 \} }\\
& \hspace{6.3cm}+ \int_{\mathbb{R}^d} p(s,w,z) \Lambda(y-z)dz\textbf{1}_{ \{ t\geq 1, s< 1 \} }+ c\textbf{1}_{ \{ t\geq 1, \ s\geq 1 \} } \Bigg\} \\
& \leq  C_4e^{-\mu_1(t+s)} \Bigg\{ c_1 (t+s)^{-\frac{\beta}{\alpha}}\textbf{1}_{ \{ t< 1, s< 1 \} }  +
c_2t^{-\frac{\beta}{\alpha}}\textbf{1}_{ \{ t< 1, s\geq 1 \} }
+c_3 s^{-\frac{\beta}{\alpha}}\textbf{1}_{ \{ t\geq 1, s< 1 \} }\ +c\textbf{1}_{ \{ t\geq 1, \ s\geq 1 \} } \Bigg\} \\
& \leq  C_5 e^{-\mu_1(t+s)}(t+s)^{-\beta/\alpha} \Bigg\{c_1\textbf{1}_{ \{ t< 1, s< 1 \} }+c_2\Big(1+\frac{s}{t}\Big)^{\beta/\alpha}\textbf{1}_{ \{ t< 1, s\geq 1 \} }+c_3\Big(1+\frac{t}{s}\Big)^{\beta/\alpha}\textbf{1}_{ \{ t\geq 1, s< 1 \} }\\
& \hspace{11.3cm} +(t+s)^{\beta/\alpha}\textbf{1}_{ \{ t\geq 1, \ s\geq 1 \} }   \Bigg\}\\
& \leq C_6 e^{-(1-\delta)\mu_1(t+s)}(t+s)^{-\beta/\alpha} \ \ \text{for all} \ \delta>0.
	\end{align*}
Again notice the use of \eqref{Pbds} in the second inequality,  the semigroup property \eqref{ChapKv} in the first integral in the third inequality and  Proposition \ref{TimSc} in the fourth inequality. 
This concludes the proof.
\end{proof}

\begin{lemma}\label{LmFact}
Suppose $a\geq 0$ and $\zeta>-1$. Then
\begin{align*}
I_n^\zeta(a,b)&:=\int_{\{a<  r_1<r_2<\cdots<r_n<b  \}}\Big[(r_2-r_1)(r_3-r_2)\cdots(b-r_n)\Big]^\zeta dr_1dr_2\cdots dr_n\\
&= \frac{\Gamma(1+\zeta)^{n+1}(b-a)^{n(1+\zeta)}}{ \Gamma\big(n(1+\zeta)+1\big)},
\end{align*}
where $\Gamma(.)$ is the Euler's gamma function.

\end{lemma}
\begin{proof}
We shall consider two cases here:\\

When $a=0$, this is  \cite[Lemma 3.5]{BalDor}.

Assume now that $a> 0$, then integrating iteratively yields:\\

 starting with

\begin{align*}
	\int_a^{r_2} (r_2-r_1)^\zeta dr_1= \frac{(r_2-a)^{1+\zeta}}{1+\zeta}.
\end{align*}
Next,
\begin{align*}
\int_a^{r_3} (r_2-a)^{1+\zeta}(r_3-r_2)^\zeta dr_2= & \int_0^{r_3-a} r_2^{1+\zeta}(r_3-a-r_2)^\zeta dr_2\\
=& (r_3-a)^{2(1+\zeta)}B\Big((1+\zeta)+1, \zeta+1\Big),
\end{align*}
where we have used successively the change of variables $r_2\rightarrow r_2-a$ and  $r_2\rightarrow \frac{r_2}{r_3-a}$ and $B(.,.)$ is the Euler's Beta function, i.e

\begin{align*}
B(c,d)=\int_0^1 u^{c-1}(1-u)^{d-1}du, \ \ c>0, d>0.
\end{align*}
Continuing this way, we end up with

\begin{align*}
I_n^\zeta(a,b)= &\frac{1}{1+\zeta}\Bigg[B\Big((1+\zeta)+1, \zeta+1\Big)B\Big(2(1+\zeta)+1, \zeta+1\Big)\cdots B\Big((n-2)(1+\zeta), \zeta+1\Big)\\
& \hspace{7.5cm}\times \int_a^b(r_n-a)^{(n-1)(1+\zeta)}(b-r_n)^\zeta dr_n\Bigg]\\
&= \frac{(b-a)^{n(1+\zeta)}}{1+\zeta}\Bigg[B\Big((1+\zeta)+1, \zeta+1\Big)B\Big(2(1+\zeta)+1, \zeta+1\Big)\\
& \hspace{4cm}\cdots B\Big((n-2)(1+\zeta), \zeta+1\Big)
 B\Big((n-1)(1+\zeta), \zeta+1\Big)\Bigg].
\end{align*}
The fact that $\Gamma(z+1)=z\Gamma(z)$ for all $z>0$ together with $B(c,d)=\frac{\Gamma(c)\Gamma(d)}{\Gamma(c+d)}$ concludes the proof.
\end{proof}
\vspace{0.5cm}

The following result is essential for the proof of the lower bound in Theorem \ref{thm1}.
\begin{proposition}\label{ProplowBd}
	Fix $\epsilon>0$. Let $u$ be the solution of \eqref{mainEq} with $\gamma=\delta_0.$ Then for all $x\in D_\epsilon$, we have
	$$
		\E|u_t(x)|^2 \geq ce^{-2\mu_1t}\sum_{n=1}^{\infty} \Big(C\xi l_\sigma \Big)^{2n}\Bigg(\frac{t^{n}}{n!}\Bigg)^{\big(\frac{\alpha-\beta}{\alpha}\big)},
	$$
	for some  positive constants $c$ and $C= C(\alpha,\beta,d)$
\end{proposition}
\begin{proof}
By  squaring the mild solution  \eqref{mlSol}, we get
\begin{align*}
	\E|u_t(x)|^2=& \big(\mathcal{G}u_0\big)_t^2(x)+\xi^2\int_0^t\int_{D^2}p_D(t-s,x,y)p_D(t-s,x,z)\E\big|\sigma\big(u_s(y)\big)\sigma\big(u_s(z)\big)\big|\\
	& \hspace{8cm}\times \Lambda(y-z)dydzds.
\end{align*}
Now using Assumption \ref{hyp3}, we get

\begin{align*}
		\E|u_t(x)|^2\geq  \big(\mathcal{G}u_0\big)_t^2(x)+\xi^2l_\sigma^2\int_0^t\int_{D_\epsilon^2}p_D(t-s,x,y)p_D(t-s,x,z)\E\big|u_s(y)u_s(z)\big|\Lambda(y-z)dydzds.
\end{align*}

 But we also have from the mild solution and Assumption \ref{hyp3} that

 \begin{align*}
 	\E\big|u_s(y)u_s(z)\big|\geq& \E\big[u_s(y)u_s(z)\big]\\ \geq & \big(\mathcal{G}u_0\big)_s(y)\big(\mathcal{G}u_0\big)_s(z)
 	+\xi^2l_\sigma^2\int_0^s\int_{D_\epsilon^2} p_D(s-s_1,y,y_1)p_D(s-s_1,z,z_1)\\
 	& \hspace{5cm}\times \E\big[u_{s_1}(y_1)u_{s_1}(z_1)\big]
 	\times \Lambda(y_1-z_1)dy_1dz_1ds_1.
 \end{align*}
Thus, combining this inequality with the previous one, we get

\begin{align*}
&\E|u_t(x)|^2\geq  \big(\mathcal{G}u_0\big)_t^2(x)+\xi^2l_\sigma^2
\int_0^t\int_{D_\epsilon^2}p_D(t-s,x,y)p_D(t-s,x,z)
\Big[\big(\mathcal{G}u_0\big)_s(y)\big(\mathcal{G}u_0\big)_s(z) \Big]\\
&\times\Lambda(y-z)dydzds + \Big(\xi^2l_\sigma^2\Big)^2\int_0^t\int_{D_\epsilon^2}p_D(t-s,x,y)p_D(t-s,x,z)\\
&\times\int_0^s\int_{D_\epsilon^2} p_D(s-s_1,y,y_1)p_D(s-s_1,z,z_1)
\E\big[u_{s_1}(y_1)u_{s_1}(z_1)\big] \Lambda(y_1-z_1)dy_1dz_1ds_1dydzds.
\end{align*}

Continuing this iteration and possibly relabeling the variables, we end up with
\begin{align*}
&\E|u_t(x)|^2\geq  \big(\mathcal{G}u_0\big)_t^2(x)\\
&+\sum_{n=1}^\infty\Big(\xi^2l_\sigma^2\Big)^n
\int_0^t\int_{D_\epsilon^2}\int_0^{s_1}\int_{D_\epsilon^2}\int_0^{s_2}
\int_{D_\epsilon^2}\cdots\int_0^{s_{n-1}}
\int_{D_\epsilon^2}\big(\mathcal{G}u_0\big)_{s_n}(y_n)\big(\mathcal{G}u_0\big)_{s_n}(z_n)\\
&\times\prod_{i=1}^n p_D(s_{i-1}-s_i,y_i,y_{i-1}) p_D(s_{i-1}-s_i,z_i,z_{i-1})\Lambda(x_i-y_i)dy_idz_ids_i\\
\geq & \big(\mathcal{G}u_0\big)_t^2(x)+\sum_{n=1}^\infty\Big(\xi^2l_\sigma^2\Big)^n
\int_0^t\int_0^{s_1}\int_0^{s_2}...\int_0^{s_{n-1}}\int_{D_\epsilon^{2n}}
\big(\mathcal{G}u_0\big)_{s_n}(y_n)\big(\mathcal{G}u_0\big)_{s_n}(z_n)\\
&\times\prod_{i=1}^n p_D(s_{i-1}-s_i,y_i,y_{i-1}) p_D(s_{i-1}-s_i,z_i,z_{i-1})\Lambda(x_i-y_i)dy_idz_ids_i
\end{align*}
where we have set $y_0:=x=:z_0$ and $s_0:=t$. Now  for $x\in D_\epsilon$, and for $i=1,2,...,n$, choose  $x_i$ and $y_i$ such that
$$
y_i\in B\Bigg(x, \frac{(s_{i-1}-s_i)^{1/\alpha}}{3}\Bigg)\cap B\Bigg(y_{i-1}, \frac{(s_{i-1}-s_i)^{1/\alpha}}{3}\Bigg)
$$
and

$$
z_i\in B\Bigg(x, \frac{(s_{i-1}-s_i)^{1/\alpha}}{3}\Bigg)\cap B\Bigg(z_{i-1}, \frac{(s_{i-1}-s_i)^{1/\alpha}}{3}\Bigg)
$$
 so that $$|z_i-z_{i-1}|<(s_{i-1}-s_i)^{1/\alpha} \ \ \text{ and} \ \ |y_i-y_{i-1}|<(s_{i-1}-s_i)^{1/\alpha}.$$ Furthermore,  $$|z_i-y_i|<(s_{i-1}-s_i)^{1/\alpha}.$$ These estimates will ensure that, for all $i=1,2,...,n$,

 $$
 p_D(s_{i-1}-s_i,y_i,y_{i-1})\geq C_1 (s_{i-1}-s_i)^{-d/\alpha} e^{-\mu_1(s_{i-1}-s_i)},
 $$

$$
p_D(s_{i-1}-s_i,z_i,z_{i-1})\geq C_2 (s_{i-1}-s_i)^{-d/\alpha} e^{-\mu_1(s_{i-1}-s_i)}
$$
  and
  $$
  \Lambda(y_i-z_i)\geq C_3 (s_{i-1}-s_i)^{-\beta/\alpha}
  $$
for some positive constants $C_1,C_2$ and $C_3$,   thanks to Proposition \ref{lwbpD} and Assumption \ref{hyp2}. Moreover, since the initial solution $u_0$ is bounded, using Lemma \ref{prop31}, we get

\begin{align*}
	\big(\mathcal{G}u_0\big)_{s_n}(y_n)\big(\mathcal{G}u_0\big)_{s_n}(z_n)\geq C_4 e^{-2\mu_1s_n}.
\end{align*}
Combining these estimates yields

\begin{align*}
	\E|u_t(x)|^2\geq& C_5e^{-2\mu_1t}\sum_{n=1}^\infty\Big(\xi^2l_\sigma^2\Big)^n\int_{\Theta_n(t)}
	\int_{A_1\times B_1}	\int_{A_2\times B_2}\\
	& \qquad\cdots	\int_{A_n\times B_n}
	\prod_{i=1}^n (s_{i-1}-s_i)^{-\beta/\alpha}(s_{i-1}-s_i)^{-2d/\alpha}dy_idz_ids_i
\end{align*}
Where $\Theta_n(t):=\Big\{ (s_0,s_1,...,s_{n-1})\in \mathbb{R}_+^n: s_0>s_1>...>s_{n-1} \Big\}$, \\ $A_i:=\Bigg\{y_i\in B\Bigg(x, \frac{(s_{i-1}-s_i)^{1/\alpha}}{3}\Bigg)\cap B\Bigg(y_{i-1}, \frac{(s_{i-1}-s_i)^{1/\alpha}}{3}\Bigg) \Bigg\}$  \\
and $B_i:=\Bigg\{z_i\in B\Bigg(x, \frac{(s_{i-1}-s_i)^{1/\alpha}}{3}\Bigg)\cap B\Bigg(z_{i-1}, \frac{(s_{i-1}-s_i)^{1/\alpha}}{3}\Bigg)  \Bigg\}$.

It is not hard to see that $\text{Volume}(A_i)\wedge\text{Volume}(B_i)\geq C_6 (s_{i-1}-s_i)^{d/\alpha}$  for all $i=1,2,...,n.$ Taking into account the latter gives

\begin{align*}
\E|u_t(x)|^2\geq& C_7e^{-2\mu_1t}\sum_{n=1}^\infty\Big(\xi^2l_\sigma^2\Big)^n\int_{\Theta_n(t)}
\prod_{i=1}^n (s_{i-1}-s_i)^{-\beta/\alpha}ds_i\\
&=  C_8e^{-2\mu_1t}\sum_{n=1}^\infty\frac{\Big(C_9\xi^2l_\sigma^2\Big)^n t^{n(1-\beta/\alpha)}}{\Gamma\Big(n(1-\beta/\alpha)+1\Big)}, \ \ C_9=C_9(\alpha,\beta),\\
\end{align*}
where we have used Lemma \ref{LmFact} with $a=0$ and $b=t$. Finally applying Stirling's approximation \ref{StirlingApp} from Proposition \ref{LmBaNus}  yields the desired result.

\vspace{0.1cm}

Armed with all the necessary tools, we can now prove our main results.
\end{proof}
\section{Proofs of  the main results}
\begin{proof}[Proof of Theorem \ref{thm1}]
For the upper bound, we combine the Burkh$\ddot{o}$lder-Davis-Gundy's,  Minkowski's and Jensen's inequalities after taking the $p^{th}$ power of the mild solution to get

 \begin{align*}
 	& \E|u_t(x)|^p \\
 	& \hspace{.5cm}\leq 2^{p-1}\Bigg\{ \Big((\mathcal{G}u_0)_t(x)\Big)^p \\
 	& \hspace{.75cm}+\xi^pz_p^p\Bigg(\int_{0}^t\int_{D\times D} p_D(t-s,x,y)p_D(t-s,x,z)\Lambda(y-z)\E|\sigma(u_s(y))\sigma(u_s(z))|  dydzds \Bigg )^{p/2} \Bigg\}
 	\\
& \hspace{.5cm} \leq	2^{p-1}\Bigg\{ \Big((\mathcal{G}u_0)_t(x)\Big)^p \\
 	& \hspace{.75cm} +\xi^pz_p^p\Bigg(\int_{0}^t \Big(\sup\limits_{y\in D}\E|\sigma(u_s(y))|^p \Big)^{2/p}\int_{D\times D} p_D(t-s,x,y)p_D(t-s,x,z)\Lambda(y-z) dydzds \Bigg )^{p/2} \Bigg\}
 \end{align*}

Where $z_p$ is as in Theorem \ref{thm1}, See for example \cite{fooVar}. Note that we have also used the following fact straight from  H$\ddot{o}$lder's inequality:
\begin{align*}
	\E|\sigma(u_s(y))\sigma(u_s(z))|& \leq \Big[\Big(\E\big|\sigma\big(u_s(y)\big)\big|^2\Big)^{1/2}\Big(\E\big|\sigma\big(u_s(z)\big)\big|^2\Big)^{1/2} \Big]\\
	& \leq \sup\limits_{y\in D}\E|\sigma(u_s(y))|^2.
\end{align*}

Because $u_0$ is bounded,  using Assumption \ref{hyp3} and Lemma \ref{prop32}, we get

\begin{align*}
	& \int_{0}^t\Big(\sup\limits_{y\in D}\E|\sigma(u_s(y))|^p \Big)^{2/p} \int_{D\times D} p_D(t-s,x,y)p_D(t-s,x,z)\Lambda(y-z) dydzds \\
	&\hspace{2cm}\leq  L_\sigma^2 \int_{0}^t \Big(\sup\limits_{y\in D}\E|u_s(y)|^p \Big)^{2/p}\int_{D\times D} p_D(t-s,x,y)p_D(t-s,x,z)\Lambda(y-z) dydzds\\
	& \hspace{2cm}\leq  L_\sigma^2 \int_0^t f(s)e^{-(2-\delta)\mu_1(t-s)}(t-s)^{-\beta/\alpha} ds,
\end{align*}

where $f(t):=\Big(\sup\limits_{x\in D}\E|u_t(x)|^p \Big)^{2/p}$.
	 Thus, defining a new function $F(t):= e^{(2-\delta)\mu_1t} f(t)$, we get for all $t>0,$
	
	$$
	F(t)\leq c_1+c_2\xi^2z_p^2\int_0^t F(s)(t-s)^{-\beta/\alpha} ds.
	$$
	Finally applying Proposition \ref{fooliuProp25} with $\rho=1-\beta/\alpha$ yields the desired  upper bound.
	
	\vspace{0.5cm}
	For the lower bound, we combine Proposition \ref{ProplowBd} and  Proposition \ref{lmBlJlis} with $\upsilon=\frac{\alpha-\beta}{\alpha}>0$,  together with Jensen's inequality  to get the desired  bound.
\end{proof}

\vspace{0.5cm}

\begin{proof}[Proof of Theorem \ref{thm2}]
	The solution to \eqref{mainEq} (when $\sigma=Id$ ) when it exists, has the following Wiener-chaos expansion in $L^2(\Omega):$
	\begin{equation}\label{Wien}
		u_t(x)= \sum_{n=0}^{\infty}\xi^n I_n\Big(h_n(.,t,x)\Big),
	\end{equation}
	where $I_0$ is the identity map on $\mathbb{R}$ and $I_n$ denotes the multiple Wiener integral with respect to $F$ in $\mathbb{R}_+^n\times D^n$ for any $n\geq 1$,  and for any $(t_1,...,t_n) \in \mathbb{R}_+^n$, $x_1,...,x_n\in D$, and for each $(t,x)$, $h_n(., t,x)$ is a symmetric element in $\mathcal{H}^{\otimes n}.$ To find an explicit expression for the kernels, we follow ideas from \cite[Section 4.1]{HuNual2} as follow:
	
	Substituting equation \eqref{Wien} into the Skorohod integral in equation \eqref{Skhd}, we get
	
	\begin{align*}
	\int_0^t\int_{D}p_D(t-s,x,y)u_s(y) F(\delta s,\delta y)=& \sum_{n=0}^\infty \int_0^t\int_{D}I_n\big(p_D(t-s,x,y)h_n(., s,y)\big) F(\delta s,\delta y)\\
	=& \sum_{n=0}^\infty I_{n+1}\Big(\widetilde{p_D(t-s,x,y)h_n(., s,y)}\Big),
	\end{align*}
where $\widetilde{p_D(t-s,x,y)h_n(., s,y)}$	is the symmetrization of the function\\ $p_D(t-s,x,y)h_n(s_1,x_1, s_2,x_2,\cdots,s_n,x_n, s,y)$ in the variables $(s_1,x_1), (s_2,x_2), \cdots, (s_n,x_n), (s,y), $ i.e

\begin{align*}
&\widetilde{p_D(t-s,x,y)h_n(., s,y)}= \frac{1}{n+1}\Big[ p_D(t-s,x,y)h_n(s_1,x_1, s_2,x_2,\cdots, s_n,x_n,s,y)\\
+& \sum_{i=1}^n p_D(t-s_j,x,y_j)h_n(s_1,x_1, s_2,x_2,\cdots,s_{j-1},x_{j-1}, s,y, s_{j+1},x_{j+1}, \cdots, s_n,x_n, s_j,x_j) \Big].
\end{align*}
Hence, equation \eqref{Skhd} is equivalent to equation \eqref{Wien} with $h_0(t,x)=(\mathcal{G}u_0)_t(x) $ and
\begin{equation}
h_{n+1}(.,t,x)= \widetilde{p_D(t-s,x,y)h_n(., s,y)}.
\end{equation}
The adaptability property of the random field $u$ implies that $h_n(s_1,x_1, s_2,x_2,\cdots, s_n,x_n,s,y)= 0$ if $s_j>t$ for some $j.$ This leads to the following formula for the kernels $h_n$, for all $n\geq 1$
\begin{equation}\label{kCh}
h_n(t_1,x_1,...,t_n,x_n,t,x)=\frac{1}{n!}\prod_{i=1}^n p_D(t_{\tau(i+1)}-t_{\tau(i)},x_{\tau(i+1)}, x_{\tau(i)}) (\mathcal{G}u_0)_{t_{\tau(1)}}(x_{\tau(1)}),
\end{equation}
where  $\tau$ denotes  the permutation of $\{1,2,\cdots, n\}$ such that $0<t_{\tau(1)}<t_{\tau(2)}<\cdots<t_{\tau(n)}<t,$ with $t_{\tau(n+1)}:=t$ and $x_{\tau(n+1)}:=x.$

This shows that there exists a unique solution to equation \eqref{Skhd} and the kernels of its Wiener-chaos expansion are given by \eqref{kCh}. In order to show the existence of a solution, it suffices to check that the kernels defined in \eqref{kCh} determine an adapted random field satisfying the conditions of Definition \ref{defskhd}. This is equivalent to show that for all $(t, x)$ we have

\begin{equation}\label{ExSolSkd}
\sum_{n=0}^{\infty}\xi^{2n} n!{\|h_n(.,t,x)\|}_{\mathcal{H}^{\otimes n}}^2<\infty.
\end{equation}
In which case,
	$$
	\E|u_t(x)|^2= \sum_{n=0}^{\infty}\xi^{2n} n!{\|h_n(.,t,x)\|}_{\mathcal{H}^{\otimes n}}^2.
	$$
	
To show \eqref{ExSolSkd}, we start with
	\begin{equation}\label{L2NormH}
	\begin{split}
&	n!{\|h_n(.,t,x)\|}_{\mathcal{H}^{\otimes n}}^2
=\frac{1}{n!}\int_{[0,t]^{2n}}\int_{D^{2n}} \prod_{i=1}^n p_D(t_{\tau(i+1)}-t_{\tau(i)},x_{\tau(i+1)}, x_{\tau(i)}) (\mathcal{G}u_0)_{t_{\tau(1)}}(x_{\tau(1)})\\
& \times \prod_{i=1}^n p_D(s_{\iota(i+1)}-s_{\iota(i)},y_{\iota(i+1)}, y_{\iota(i)}) (\mathcal{G}u_0)_{s_{\iota(1)}}(y_{\iota(1)}) \prod_{i=1}^n \gamma(t_i-s_i)\prod_{i=1}^n \Lambda(x_i-y_i)	 d\textbf{x}d\textbf{y}d{\bf t}d{\bf s}
	\end{split}
	\end{equation}

\noindent Notice that, for simplicity,  we write  $d{\bf t}=dt_1...dt_n$, $d{\bf s}=ds_1...ds_n$, $d\textbf{x}= dx_1...dx_n$ and $d\textbf{y}= dy_1...dy_n$ and the permutations $\tau$ and $\iota$ of $\{1,2,\cdots,n\}$ are such that

$$
0<t_{\tau(1)}<t_{\tau(2)}<\cdots<t_{\tau(n)}<t \ \ \ \text{and} \ \ \ 0<s_{\iota(1)}<s_{\iota(2)}<\cdots<s_{\iota(n)}<t
$$
 with $t_{\tau(n+1)}= s_{\iota(n+1)}= t$ and $x_{\tau(n+1)}=y_{\iota(n+1)}=x.$\\

Now apply Lemma \ref{lm2} iteratively to get
	
	\begin{align*}
	 n!{\|h_n(.,t,x)\|}_{\mathcal{H}^{\otimes n}}^2 &\leq\frac{1}{n!}
		C_1 e^{-(2-\delta)\mu_1 t}\\
		&\times\int_{[0, t]^{2n}}
		\prod_{i=1}^n \gamma(t_i-s_i)\prod_{i=1}^n
		\Big(t_{\tau(i+1)}+s_{\iota(i+1)}-(t_{\tau(i)}+s_{\iota(i)})\Big)^{-\beta/\alpha}d{\bf t}d{\bf s}.	
	\end{align*}
 It follows that 		
	\begin{align*}
	 n!{\|h_n(.,t,x)\|}_{\mathcal{H}^{\otimes n}}^2 \leq\frac{1}{n!}
	C_1 e^{-(2-\delta)\mu_1 t}\int_{[0, t]^{2n}}
	\prod_{i=1}^n \gamma(t_i-s_i)\prod_{i=1}^n
	\Big(t_{\tau(i+1)}-t_{\tau(i)}\Big)^{-\beta/\alpha}d{\bf t}d{\bf s}.	
	\end{align*}
 We first take care of the integrals $\int_{ [0, t]^{n}}
	\prod_{i=1}^n \gamma(t_i-s_i)d\textbf{s}.$

\noindent For $i=1,2,\cdots n$,
\begin{align*}
\int_0^t \gamma(t_i-s_i)ds_i=& \int_{t_i-t}^{t_i} \gamma(r)dr\\
\leq & \int_{-t}^0 \gamma(r)dr+ \int_0^t\gamma(r)dr.
\end{align*}

	since $\gamma$ satisfies Assumption \ref{hyp1}. Therefore,  setting $\kappa(t):= \int_{0}^t [\gamma(-r)+ \gamma(r)]dr$, we have
	
	\begin{align*}
		\int_{[0, t]^{2n}}
		\prod_{i=1}^n \gamma(t_i-s_i)\prod_{i=1}^n
		\Big(t_{\tau(i+1)}-t_{\tau(i)}\Big)^{-\beta/\alpha}d{\bf t}d{\bf s}& \leq \kappa(t)^n\int_{[0, t]^{n}}\prod_{i=1}^n
	\Big(t_{\tau(i+1)}-t_{\tau(i)}\Big)^{-\beta/\alpha}d{\bf t}\\
		& \leq  \frac{\kappa(t)^n n!C_1^{n+1}t^{n(1-\beta/\alpha)}}{\Gamma\big(n(1-\beta/\alpha)+1\big)},  \ \ \ C_1=C_1(\alpha,\beta).
	\end{align*}
Note the use of Lemma \ref{LmFact} with $a=0$ and $b=t$ in the second inequality.
 Now using Stirling's approximation \eqref{StirlingApp} from Proposition \ref{LmBaNus}, we have for  $n=0, 1,2,\cdots$
\begin{align*}
	\sum_{n=0}^{\infty}\xi^{2n} n!{\|h_n(.,t,x)\|}_{\mathcal{H}^{\otimes n}}^2 \leq 	C_2 e^{-(2-\delta)\mu_1 t}\sum_{n\geq 0}\frac{\Big(C_1\xi^{2}\kappa\Big)^n t^{n(1-\beta/\alpha)}}{(n!)^{1-\beta/\alpha}}.
\end{align*}
This proves \eqref{ExSolSkd}.

\noindent Moreover,  using Minkowski's inequality and the equivalence of norms in a fixed Wiener chaos space, it follows that
\begin{align*}
	\Big(\E|u_t(x)|^p\Big)^{1/p}& \leq \sum_{n=0}^\infty (p-1)^{n/2}\xi^n \Big( n!{\|h_n(.,t,x)\|}_{\mathcal{H}^{\otimes n}}^2\Big)^{1/2}\\
	& \leq C_2 e^{-(1-\delta)\mu_1t}\sum_{n=0}^\infty (p-1)^{n/2}\xi^n \frac{\kappa^{n/2}t^{n(\alpha-\beta)/2\alpha}}{(n!)^{(\alpha-\beta)/2\alpha}}
\end{align*}

\noindent Finally, using Proposition \ref{LmBaNus} with $\nu=(\alpha-\beta)/2\alpha$  yields the desired upper bound in Theorem \ref{thm2}.


\noindent We now turn our attention to the proof of  the lower bound.
From equation \eqref{L2NormH}, we reduce the  spatial domain of integration and also choose a specific ordering in the indexes appearing in the time and space variables (i.e $\tau(i)=\iota(i)=i, \ \ 1\leq i\leq n$) in \eqref{L2NormH} to get
\begin{align*}
& n!{\|\tilde{h}_n(.,t,x)\|}_{\mathcal{H}^{\otimes n}}^2
\geq \frac{1}{n!}\int_{ [0,t]^{2n}}\int_{D_\epsilon^{2n}} \prod_{i=1}^n p_D(t_{i+1}-t_{i},x_{i+1}, x_{i}) (\mathcal{G}u_0)_{t_{1}}(x_{1})\\
& \hspace{3cm}\times \prod_{i=1}^n p_D(s_{i+1}-s_{i},y_{i+1}, y_{i}) (\mathcal{G}u_0)_{s_{1}}(y_{1}) \prod_{i=1}^n \gamma(t_i-s_i)\prod_{i=1}^n \Lambda(x_i-y_i)	 d\textbf{x}d\textbf{y}d{\bf t}d{\bf s},
\end{align*}

 Since $x\in D_\epsilon$, for $i=1,...,n$, choose $x_{i}, y_{i} \in D_\epsilon$ satisfying:

 $$
 x_{i}\in B\Big(x,\frac{1}{3}\big(t_{i+1}-t_{i}\big)^{1/\alpha} \Big)\cap B\Big(x_{i-1}, \frac{1}{3}\big(t_{i+1}-t_{i}\big)^{1/\alpha}\Big)
 $$
 and

 $$
 y_{i}\in B\Big(x,\frac{1}{3}\big(s_{i+1}-s_{i}\big)^{1/\alpha} \Big)\cap  B\Big(y_{i-1}, \frac{1}{3}\big(s_{i+1}-s_{i}\big)^{1/\alpha}\Big),
 $$
 with $x_0:=x=:y_0.$
 Thus, for $i=1,2,...,n$, it follows that \\ $$|x_{i}-x_{i-1}|<\Big(t_{i+1}-t_{i}\Big)^{1/\alpha} \ \ \text{ and} \ \   |y_{i}-y_{i-1}|<\Big(s_{i+1}-s_{i}\Big)^{1/\alpha}.$$
 Furthermore, using the fact that for positive $a,b$ and $0<p\leq 1$, $(a+b)^p\geq c_p( a^p+b^p)$, we get $$|x_{i}-y_{i}|< c_1\Big(t_{i+1}-t_{i}+s_{i+1}-s_{i}\Big)^{1/\alpha}. $$  This gives us    $$\Lambda(x_{i}-y_{i})\geq c_2(t_{i+1}-t_{i}+s_{i+1}-s_{i})^{-\beta/\alpha}\geq c_2(t_{i+1}-t_{i}+t)^{-\beta/\alpha}, \ \ \text{since}\ \ s_{i+1}-s_i\leq t \ \ \text{for all} \  i=1,\cdots, n.$$ \\
 Now appealing to Proposition \ref{lwbpD}, for $i=1,2,\cdots,n$,  we get
 $$p_D\big(t_{i+1}-t_{i},x_{i},x_{i-1}\big)\geq c_3 \Big(t_{i+1}-t_{i}\Big)^{-d/\alpha} e^{-\mu_1\big(t_{i+1}-t_{i}\big)}$$ and  $$p_D\big(s_{i+1}-s_{i},y_{i},y_{i-1}\big)\geq c_4 \Big(s_{i+1}-s_{i}\Big)^{-d/\alpha} e^{-\mu_1\big(s_{i+1}-s_{i}\big)}.$$
 Next, denoting  $$\mathcal{A}_i:=\Big\{x_{i}\in B\Big(x,\frac{1}{3}\big(t_{i+1}-t_{\tau(i)}\big)^{1/\alpha} \Big)\cap B\Big(x_{i-1}, \frac{1}{3}\big(t_{i+1}-t_{i}\big)^{1/\alpha}\Big)\Big\}$$ and
 $$\mathcal{B}_i:=\Big\{  y_{i}\in B\Big(x,\frac{1}{3}\big(s_{i+1}-s_{i}\big)^{1/\alpha} \Big)\cap  B\Big(y_{i-1}, \frac{1}{3}\big(s_{i+1}-s_{i}\big)^{1/\alpha}\Big)\Big\},$$ it follows that

 \begin{align*}
 &n!{\|\tilde{h}_n(.,t,x)\|}_{\mathcal{H}^{\otimes n}}^2   \geq \frac{1}{n!}C_4 e^{-2\mu_1 t}
 \int_{[0,t]^{2n}}\int_{\mathcal{A}_1\times \mathcal{B}_1} \int_{\mathcal{A}_2\times \mathcal{ B}_2}...\int_{\mathcal{A}_n\times \mathcal{B}_n}\prod_{i=1}^n\gamma(t_i-s_i) \\
 &\hspace{3cm}\times\prod_{i=1}^n \Big(t+t_{i+1}-t_{i}\Big)^{-\beta/\alpha}  \Big(t_{i+1}-t_{i}\Big)^{-d/\alpha}\big(s_{i+1}-s_{i}\big)^{-d/\alpha} d\textbf{x}d\textbf{y}d\textbf{t}d\textbf{s}.
 \end{align*}
 Note the use of Lemma \ref{prop31}-a) and the fact that $u_0$ is bounded in the above display.\\

 It's not hard to see that, for  $i=1,2,...,n$, $$\text{Volume}(\mathcal{A}_i)\geq C_i\Big(t_{i+1}-t_{i}\Big)^{d/\alpha} \ \  \ \ \text{and} \ \ \ \  \text{Volume}(\mathcal{B}_i)\geq c_i\big(s_{i+1}-s_{i}\big)^{d/\alpha}$$ for some positive constants $c_i$ and $C_i$ independent of $s_i$ and $t_i$ respectively.  This will ensure that

 \begin{align*}
 n!{\|\tilde{h}_n(.,t,x)\|}_{\mathcal{H}^{\otimes n}}^2   \geq \frac{1}{n!}C_5e^{-2\mu_1 t}\int_{[0,t]^{2n}}\prod_{i=1}^n\gamma(t_i-s_i)\prod_{i=1}^n \Big(t+t_{i+1}-t_{i}\Big)^{-\beta/\alpha} d\textbf{t}d\textbf{s}.
 \end{align*}
 The iterative integrals with time correlation integrands deserve a particular attention. We take care of them first. Assume $t_1\geq t/2$,
 \begin{align*}
 \int_{0}^t \gamma\big(t_{1}-s_{1}\big)ds_{1}= &  \int_{t_{1}-t}^{t_1}\gamma(r)dr\\
 \geq & \int_0^{t_1} \gamma(r)dr,
 \end{align*}
 since $t_{1}-t<0$.

 Next,

 \begin{align*}
 \int_{0}^t \gamma(t_{2}-s_{2})ds_{2}= &  \int_{t_{2}-t}^{t_{2}}\gamma(r)dr\\
 \geq & \int_0^{t_2} \gamma(r)dr,
 \end{align*}
 where, again  $t_{2}-t<0$.\\

 Continuing this way, the last integral in the iteration is
 \begin{align*}
 \int_0^{t}\gamma(t_{n}-s_{n})ds_{n} & = \int_{t_{n}-t}^{t_{n}}\gamma(r)dr
 \end{align*}
 since $t_{n}-t<0$. \\

 Therefore

 $$
 \int_{[0,t]^{n}}\prod_{i=1}^n\gamma(t_i-s_i) d\textbf{s}\geq \Bigg(\int_0^{t_1} \gamma(r)dr\Bigg)^n
 $$

 Thus, choosing $t_1\geq t/2$ and setting $$\eta(t):= \int_{0}^{t/2}\gamma(r)dr,$$ we get


 \begin{align*}
 n!{\|\tilde{h}_n(.,t,x)\|}_{\mathcal{H}^{\otimes n}}^2  & \geq \frac{1}{n!}C_8 e^{-2\mu_1 t}\eta(t)^n\int_{
 	[t/2,\ t]^n} \prod_{i=1}^n \Big(t+ t_{i+1}-t_{i}\Big)^{-\beta/\alpha} d\textbf{t}\\
 & \geq \frac{1}{n!}C_8 e^{-2\mu_1 t}\eta(t)^n t^{n\big(\frac{\alpha-\beta}{\alpha}\big)}	\int_{[1/2, 1]^n} \prod_{i=1}^n \Big(1+ t_{i+1}-t_{i}\Big)^{-\beta/\alpha} d\textbf{t}\\
 & \geq \frac{1}{n!}C_9 e^{-2\mu_1 t}\eta(t)^n t^{n\big(\frac{\alpha-\beta}{\alpha}\big)}	\int_{[1/2, 1]^n} \prod_{i=1}^n \Big(1+ (t_{i+1}-t_{i})^{\beta/\alpha}\Big)^{-1} d\textbf{t}
 \end{align*}

 Next we  use the Laplace transform of the Mittag-Leffler function  for any $\nu\in (0,1)$, see for example \cite[(1.80), P. 21]{Podlny}. It is given by
 \begin{equation}\label{LapMitag}
 \int_0^\infty e^{-pt}E_\nu(-t^\nu)dt=p^{\nu-1}(1+p^\nu)^{-1}.
 \end{equation}
 Here, $E_\nu(\cdot)$ represents the one-parameter Mittag-Leffler function defined by $$E_\nu(z)=\sum_{k=0}^\infty \frac{z^n}{\Gamma(\nu k+1)}.$$
 Using  \eqref{LapMitag}, we get

 \begin{equation}
 \begin{split}
 \prod_{i=1}^n \Big(1+ (t_{i+1}-t_{i}) ^{\frac{\beta}{\alpha}}\Big)^{-1}&=\bigg[\prod_{i=1}^n (t_{i+1}-t_{i}) ^{1-\frac{\beta}{\alpha}}\bigg]\int_{[0,\infty)^n} \prod_{i=1}^n e^{-(t_{i+1}-t_{i})l_i}E_{\beta/\alpha}\Big(-l_i^\frac{\beta}{\alpha}\Big)dl_i.\\
 \end{split}
 \end{equation}
 Now fix $\epsilon\in (0,1)$. Then
 \begin{align*}
 \int_{[0,\infty)^n} \prod_{i=1}^n e^{-(t_{i+1}-t_{i})l_i}E_{\beta/\alpha}\Big(-l_i^{\beta/\alpha}\Big)dl_i
 &\geq\int _{[0, \epsilon]^n}\prod_{i=1}^n e^{-(t_{i+1}-t_{i})l_i}E_{\beta/\alpha}\Big(-l_i^{\beta/\alpha}\Big)dl_i\\
 &\geq \int _{[0, \epsilon]^n}e^{-\frac{1}{2}\sum_{i=1 }^n l_i} \prod_{i=1}^n E_{\beta/\alpha}\Big(-l_i^{\beta/\alpha}\Big)dl_i\\
 &\geq \int _{[0, \epsilon]^n}e^{-\frac{n\epsilon}{2}}  \prod_{i=1}^n\Big(Ce^{-c_{\alpha,\beta}l_i^{\beta/\alpha}}\Big)dl_1\cdots dl_n\\
 &\geq \int _{[0, \epsilon]^n}e^{-\frac{n\epsilon}{2}}  \Big(Ce^{-c_{\alpha,\beta}\epsilon^{\beta/\alpha}}\Big)^{n}dl_1\cdots dl_n\\
 &\geq C_{\alpha,\beta,\epsilon}^n,
 \end{align*}
 where $C_{\alpha,\beta,\epsilon}>0$ is a constant independent of $n$ and $t$.
Note that we have used the asymptotic approximation of the Mittag-Leffler near the origin in the third inequality above, see for example \cite[(3.4)]{Mainardi} for a reference.

 Therefore, we get
 \begin{align*}
 n!{\|\tilde{h}_n(.,t,x)\|}_{\mathcal{H}^{\otimes n}}^2& \geq \frac{1}{n!}C_9 e^{-2\mu_1 t}\eta(t)^n t^{n\big(\frac{\alpha-\beta}{\alpha}\big)}C_{\alpha,\beta,\epsilon}^n	\int_{[1/2, 1]^n} \prod_{i=1}^n (t_{i+1}-t_{i}) ^{-\beta/\alpha} d\textbf{t},
 \end{align*}

Next, we use Lemma \ref{LmFact} with $a=1/2$ and $b=1$ to get
\begin{align*}
 n!{\|\tilde{h}_n(.,t,x)\|}_{\mathcal{H}^{\otimes n}}^2  & \geq \frac{1}{n!}C_9 e^{-2\mu_1 t}\eta(t)^nt^{n\big(\frac{\alpha-\beta}{\alpha}\big)} \frac{n! C_{10}^{n+1}}{\Gamma\big(n(1-\beta/\alpha)+1\big)}\\
 &\geq C_{11} e^{-2\mu_1 t}\eta(t)^nt^{n\big(\frac{\alpha-\beta}{\alpha}\big)} \frac{ C_{12}^{n}}{{(n!)}^{1-\beta/\alpha}}.
\end{align*}

\noindent Note the  use of Stirling's approximation \eqref{StirlingApp} from Proposition \ref{LmBaNus} in the last inequality.Therefore,  for  $n=0, 1,2,\cdots$, we have

\begin{align*}
\E|u_t(x)|^2\geq 	C_{11} e^{-2\mu_1 t}\sum_{n= 0}^\infty\frac{\Big(C_{12}\xi^{2}\eta(t) t^{1-\beta/\alpha}\Big)^n}{(n!)^{1-\beta/\alpha}}.
\end{align*}
Finally, Proposition \ref{lmBlJlis} and an application of Jensen's inequality conclude the proof.

\end{proof}

\newpage
\section{Appendix}\label{Appdx}
We compile in this section the proof of existence and uniqueness of a solution to \eqref{mainEq}  driven by a space-colored noise, i.e $\gamma=\delta_0$ and some results from other authors that we have used in our paper.
\begin{proposition}\label{ExUnq}
Consider equation \eqref{mainEq} driven by a space-colored noise and suppose that the space correlation function satisfies condition \eqref{DalCond}, then there exists a unique solution to \eqref{mainEq}.
\end{proposition}
\begin{proof}
The proof for the existence of a solution of equation \eqref{mainEq} driven by a space-colored  noise follows a Picard's iteration scheme. We follow ideas from \cite{DaKoMuNuXi, FooLiu}. The details are provided below:

Let $u_t^0(x)=\mathcal{G}_D u_0(x)$ and for $n\geq 1,$

$$
u_t^{n+1}(x)= \Big(\mathcal{G}_D u_0\Big)_t(x)+ \xi \int_0^t \int_D p_D(t-s,x,y)\sigma\big( u_s^n(y)\big) F(ds,dy).
$$
The stochastic integral is well defined even when the correlation function is restricted to $D.$ This fact actually follows from Walsh \cite{WalshB}. Let $D_n(t,x)= u_t^{n+1}(x)-u_t^n(x)$. It follows that

$$
D_n(t,x)= \xi \int_0^t \int_D p_D(t-s,x,y)\Big[\sigma\big( u_s^n(y)\big)-\sigma\big( u_s^{n-1}(y)\big)\Big] F(ds,dy).
$$
Now using Burkh\"older's inequality and Assumption \ref{hyp3}, we get

\begin{align*}
\mathbb{E}\big|D_n(t,x)\big|^p\leq C_p \xi^p L_\sigma^p & \Bigg[ \int_0^t \int_{D^2} p_D(t-s,x,y)p_D(t-s,x,z) \Lambda(y-z)\\
& \qquad\qquad\qquad\qquad\qquad\times\sup\limits_{y\in D}\mathbb{E}\big[ D_{n-1}(s,y)\big]^2 dydzds\Bigg]^{p/2}.
\end{align*}
Now let $$H_n(t,x)= \sup\limits_{x\in D}\sup\limits_{0<t\leq T}\mathbb{E}\big|D_n(t,x)\big|^p \ \  \text{and} \ \ F(t)= \int_0^t\int_{D^2}p_D(t-s,x,y)p_D(t-s,x,z)\Lambda(y-z)dydzds.$$  Note that

\begin{align*}
F(t)\leq &\int_0^\infty\int_{\mathbb{R}^d\times \mathbb{R}^d}p(t-s,x,y)p(t-s,x,z)\Lambda(y-z)dydz ds \\
\leq & c \int_{\mathbb{R}^d}\frac{\mu(\zeta)}{1+|\zeta|^\alpha}.
\end{align*}
Thus, $F(t)<\infty$ whenever Dalang's condition \eqref{DalCond} holds. Note also that, combining the semigroup property of the Dirichlet kernel with Lemma \ref{prop32},

$G(t)= \int_{D^2} p_D(t-s,x,y)p_D(t-s,x,z)dydz <\infty$  for  $0\leq  t\leq T.$
We now use H\"older's inequality to obtain

\begin{equation}\label{key}
H_n(t,x)\leq C_p \xi^p L_\sigma^p F(t)^{p/2-1}\int_0^t H_{n-1}(s) ds.
\end{equation}
Thus, by Gronwall's lemma-Lemma \ref{GronWDav}-we have $\sum_{n=1}^\infty H_n(t)<\infty.$ Therefore, $u_t^n(x)$ converges in $L^2(P)$ to some $u_t(x)$ for each $t$ and $x.$ This
also proves  that

$$
\lim\limits_{n\rightarrow\infty}\int_0^t \int_D p_D(t-s,x,y)\sigma\big( u_s^n(y)\big) F(ds,dy) = \int_0^t \int_D p_D(t-s,x,y)\sigma\big( u_s(y)\big) F(ds,dy).
$$
where the convergence holds in $L^2(P).$ This proves that $u$ is a solution to \eqref{mainEq} driven by a space-colored noise.

For the uniqueness, assume that $u$ and $v$ are both solution of \eqref{mainEq} with the same initial condition and driven by a space-colored noise satisfying the integrability condition \eqref{DalCond}. We show that one these solutions is a modification of the other.

\noindent Let $D(x,t)=u_t(x)-v_t(x).$ Then,
$$
D(t,x)= \xi \int_0^t \int_D p_D(t-s,x,y)\Big[\sigma\big( u_s(y)\big)-\sigma\big( v_s(y)\big)\Big] F(ds,dy).
$$

Using Assumption \ref{hyp3}, we get

$$
\mathbb{E}\Big|D(t,x)\Big|^2\leq \xi^2L_\sigma^2 \int_0^t\int_{D^2}p_D(t-s,x,y)p_D(t-s,x,z)\Lambda(y-z) \sup\limits_{y\in D} \mathbb{E}\Big|D(s,y)\Big|^2 dydzds.
$$
Now, set $H(t)=\sup\limits_{0<s\leq t}\sup\limits_{x\in D}\mathbb{E}\big|D(s,x)\big|^2$. It follows that

$$
H(t)\leq C\xi^2L_\sigma^2 \int_0^t H(s) (t-s)^{-\beta/\alpha} ds.
$$

Now choose and fix some $q\in (1,2)$ and let $r$ be its conjugate exponent, i.e $q^{-1}+r^{-1}=1.$ Next,  apply H\"older's inequality to find that there exists some constant $A=A_T(\xi, L_\sigma)$, such that uniformly for all $0\leq  t\leq T$,

$$
H(t)\leq A \Bigg(\int_0^t H^r(s) ds\Bigg)^{1/r}.
$$
Finally, apply Gronwall's Lemma  with $a_1=a_2=\cdots= H^r$ to find that $H(t)\equiv 0$ and this concludes the proof.
\end{proof}

\begin{proposition}\cite[Proposition $2.5$]{FooLiu}\label{fooliuProp25}
	Let $\rho>0$ and suppose that $f(t)$ is a locally integrable function satisfying
	$$
	f(t)\leq c_1 +\kappa\int_{0}^{t}(t-s)^{\rho-1}f(s)ds \ \ \ \text{for all} \ \ t>0,
	$$
	where $c_1$ is some positive constant. Then, we have
	
	$$
	f(t)\leq c_2 e^{c_3 \big(\Gamma(\rho)\kappa\big)^{1/\rho} t} \ \ \ \text{for all } \ \ t>0,
	$$
	for some positive constants $c_2$ and $c_3$.
\end{proposition}

\begin{lemma}\cite[Proposition $3.1$]{ENua}\label{prop31}
	For any $\epsilon\in (0,\frac{1}{2}),$ there exist positive constants $c_1(\epsilon)$ such that for all $x,\ w \in D_\epsilon$ and $t>0$ ,
	$$
	a) \int_{D_\epsilon}p_D(t,x,y)dy\geq c_1 e^{-\mu_1 t}.
	$$
	If we further impose  $|x-w|\leq t^{1/\alpha}$, then there exists a positive constant  $c_2(\epsilon)$  such that
	$$
	b)\int_{D_\epsilon\times D_\epsilon} p_D(t,x,y)p_D(t,w,z)\Lambda(y-z) dydz\geq
	c_2 e^{-2\mu_1 t}t^{-\beta/\alpha}.
	$$	
\end{lemma}
\begin{lemma}\cite[Proposition 3.2]{ENua}\label{prop32}
	For all $\delta>0,$ there exists $c_2(\delta)>0$ such that for all $x,\ w\in D$ and $t>0,$
	
	$$
	a)\int_D p_D(t,x,y) dy\leq c e^{-\mu_1t}
	$$
	\\
	\\
	$$
	b)\int_{D\times D} p_D(t,x,y)p_D(t,w,z)\Lambda(y-z) dydz\leq
	c_2 e^{-(2-\delta)\mu_1 t }t^{-\beta/\alpha}.
	$$
	
\end{lemma}
\begin{theorem}\cite[theorem $1.1$]{ChenKim}\label{chenkim}
	Assume $\alpha\in (0,2)$. There exists a positive constant $C$ such that for all $x,y\in D$ and $t>0$,
	\begin{align*}
	& C^{-1}e^{-\mu_1t}\Bigg[ \min\Big(1, \frac{\phi_1(x)}{\sqrt{t}}\Big)\min\Big(1, \frac{\phi_1(y)}{\sqrt{t}}\Big)\min\Big(t^{-d/\alpha}, \frac{t}{|x-y|^{\alpha+d}}\Big)\textbf{1}_{\{t<1\}}+\phi_1(x)\phi_1(y) \textbf{1}_{\{t\geq 1\}}\Bigg]\\
	& \hspace{5cm}	\leq p_D(t,x,y)\leq \\
	&	Ce^{-\mu_1t}\Bigg[ \min\Big(1, \frac{\phi_1(x)}{\sqrt{t}}\Big)\min\Big(1, \frac{\phi_1(y)}{\sqrt{t}}\Big)\min\Big(t^{-d/\alpha}, \frac{t}{|x-y|^{\alpha+d}}\Big)\textbf{1}_{\{t<1\}}+\phi_1(x)\phi_1(y) \textbf{1}_{\{t\geq 1\}}\Bigg]
	\end{align*}	
\end{theorem}
\begin{theorem}\cite[theorem $2.2$]{RiahL}\label{Riah}
Asume $\alpha=2$. Then there exist positive constants $c_1, C_1, c_2$ and $C_2$ such that for all $x,y\in D$ and $t>0$,
\begin{align*}
	C_1  \min\Bigg( 1, \frac{\phi_1(x)\phi_1(y)}{1\wedge t}\Bigg)e^{-\mu_1 t}\frac{e^{-c_1\frac{|x-y|^2}{t}}}{1\wedge t^{d/2}}
\leq  p_D(t,x,y)\leq
 C_2  \min\Bigg( 1, \frac{\phi_1(x)\phi_1(y)}{1\wedge t}\Bigg)e^{-\mu_1 t}\frac{e^{-c_2\frac{|x-y|^2}{t}}}{1\wedge t^{d/2}}
\end{align*}	
\end{theorem}

\begin{proposition}\cite[Lemma $A.1$]{BalNus2}\label{LmBaNus}
For any $\nu>0$,
\begin{align*}
	\sum_{k=0}^\infty \frac{x^k}{(k!)^\nu}\leq C_1 e^{c_1 x^{1/\nu}}, \ \ x>0,
\end{align*}
for some constants $c_1(\nu) \  \text{and }\  C_1(\nu)>0$.

Moreover,
\begin{equation}\label{StirlingApp}
\Gamma\big(n\tau+1\big)\sim C_n (n!)^{\tau}, \ \ \tau>0,
\end{equation}
where $C_n$ is such that $\lambda^{-n}\leq C_n\leq \lambda^n$  for some $\lambda(\alpha,\beta)>1$.
\end{proposition}
\begin{proposition}\cite[Lemma 5.2]{BalJolis}\label{lmBlJlis}
For any $\nu>0$,
\begin{align*}
	\sum_{k=1}^\infty \frac{x^k}{(k!)^\upsilon}\geq c_1 e^{ c_2x^{1/\upsilon}}, \ \ x>0,
\end{align*}
for some constants $c_1(\upsilon)>0$ and $c_2(\upsilon)>0$.	
\end{proposition}

\begin{proposition}\cite[Lemma 2.2]{BalNus}\label{TimSc}
	For any $t>0$ and $w\in \mathbb{R}^d$,
	\begin{align*}
	\int_{\mathbb{R}^d} e^{-t|\upsilon|^{\alpha}}|\omega-\upsilon|^{-d+\beta} d\upsilon\leq K_{d,\alpha,\beta} t^{-\beta/\alpha},
	\end{align*}
	where
	$$
	 K_{d,\alpha,\beta}:= \sup\limits_{w\in \mathbb{R}^d}\int_{\mathbb{R}^d}\frac{|\upsilon|^{-d+\beta}}{1+|\omega-\upsilon|^{\alpha}} d\upsilon.
	$$
\end{proposition}

\begin{lemma}\cite[Lemma 6.5]{DaKoMuNuXi}\label{GronWDav}
	Suppose $a_1, a_2, \cdots: [0 , T] \rightarrow \mathbb{R_+}$ are measurable and non-decreasing. Suppose also that there exist a constant $A$ such that for all integers $n \geq 1,$  and all $t \in [0, T],$
	
	$$
	a_{n+1}(t)\leq A\int_0^t a_n(s) ds.
	$$
	
	Then,
	
	$$
	a_n(t)\leq a_1(T)\frac{(At)^{n-1}}{(n-1)!} \ \ \text{for all} \ n\geq 1 \ \ \text{and} \ t\in[0,T].
	$$
\end{lemma}

{\bf Acknowledgement:}  The authors would like to sincerely thank the editor and  an unanimous   referee for useful
	comments that improved the paper very much.

\end{document}